\documentclass[12pt]{article}
\usepackage{microtype}
\usepackage{color}
\usepackage[small,raggedright]{titlesec}
\usepackage{stmaryrd,amssymb,amsmath,amsthm,mathtools}
\usepackage{enumerate}
\usepackage{hyperref}
\usepackage{cite}

\hypersetup{colorlinks=true,linkcolor=blue,citecolor=blue,
	        pdfinfo={
	        Title   = {Selfadjoint extensions of the multiplication 
	        		   operator in dB spaces as singular rank-one 
	        		   perturbations},
	        Author  = {Luis O. Silva, Julio H. Toloza},
	        Subject = {Manuscript}}
	       }

\usepackage[pdftex,letterpaper,hmargin=3.7cm,vmargin=3.5cm]{geometry}

\usepackage{setspace}
\usepackage{graphicx}
\newtheorem{theorem}{Theorem}[section]
\newtheorem{proposition}[theorem]{Proposition}
\newtheorem{lemma}[theorem]{Lemma}

\theoremstyle{definition}

\newtheorem{example}[theorem]{Example}

\theoremstyle{remark}
\newtheorem{remark}{Remark}

\newcommand{\abs}[1]{\left\lvert #1 \right\rvert}
\newcommand{\abss}[1]{\big\lvert #1 \big\rvert}
\newcommand{\norm}[1]{\left\lVert #1 \right\rVert}
\newcommand{\norms}[1]{\big\lVert #1 \big\rVert}

\newcommand{\inner}[2]{\left\langle#1,#2\right\rangle}
\newcommand{\dual}[2]{\left\langle#1,#2\right\rangle}

\newcommand{\fF}{\mathfrak{F}}

\newcommand{\fI}{\mathfrak{I}}

\newcommand{\cD}{\mathcal{D}}
\newcommand{\cF}{\mathcal{F}}

\newcommand{\cH}{\mathcal{H}}
\newcommand{\cB}{\mathcal{B}}

\newcommand{\cM}{\mathcal{M}}

\newcommand{\cPW}{\mathcal{PW}}
\newcommand{\cHB}{\mathcal{HB}}

\newcommand{\R}{{\mathbb R}}
\newcommand{\C}{{\mathbb C}}
\newcommand{\N}{{\mathbb N}}

\newcommand{\cc}[1]{\overline{#1}}

\renewcommand{\restriction}{\mathord{\upharpoonright}}

\DeclareMathOperator{\re}{re}
\DeclareMathOperator{\im}{im}
\DeclareMathOperator{\dom}{dom}

\DeclareMathOperator{\spec}{spec}
\DeclareMathOperator*{\lspan}{span}

\DeclareMathOperator{\assoc}{assoc}

\begin{document}

\begin{titlepage}
\title
{\vspace{-1cm}
Selfadjoint extensions of the multiplication operator in de Branges
spaces as singular rank-one perturbations
\footnotetext{%
Mathematics Subject Classification(2010):
Primary
46E22, 
47A70; 
Secondary
47A55,  
47B15,  
47B25.  
}
\footnotetext{%
Keywords: de Branges spaces; singular rank-one perturbations; scale of
Hilbert spaces.
}
\\[2mm]}
\author{
\textbf{Luis O. Silva}\thanks{Partially supported by SEP-CONACYT (Mexico) under 
	Grant CB-2015 254062}
\\
\small Departamento de F\'{i}sica Matem\'{a}tica\\[-1.6mm]
\small Instituto de Investigaciones en Matem\'{a}ticas Aplicadas y
	en Sistemas\\[-1.6mm]
\small Universidad Nacional Aut\'{o}noma de M\'{e}xico\\[-1.6mm]
\small C.P. 04510, M\'{e}xico D.F.\\[-1.6mm]
\small \texttt{silva@iimas.unam.mx}
\\[4mm]
\textbf{Julio H. Toloza}\thanks{Partially supported by CONICET
	(Argentina) under Grant PIP 11220150100327CO}
\\
\small Instituto de Matem\'{a}tica de Bah\'{i}a Blanca\\[-1.6mm]
\small Universidad Nacional del Sur\\[-1.6mm]
\small Consejo Nacional de Investigaciones Cient\'{i}ficas y
       T\'{e}cnicas\\[-1.6mm]
\small Departamento de Matem\'{a}tica\\[-1.6mm]
\small Av. Alem 1253, B8000CPB Bah\'{i}a Blanca, Argentina\\[-1.6mm]
\small \texttt{julio.toloza@uns.edu.ar}}

\date{}
\maketitle

\begin{center}
\begin{minipage}{5in}
  \centerline{{\bf Abstract}} \bigskip We derive a
  description of the family of canonical selfadjoint extensions of the
  operator of multiplication in a de Branges space in terms of
  singular rank-one perturbations using distinguished elements from
  the set of functions associated with a de Branges space.  The scale
  of rigged Hilbert spaces associated with this construction is
  also studied from the viewpoint of de Branges's theory.
\end{minipage}
\end{center}

\bigskip
\thispagestyle{empty}

\end{titlepage}

\section{Introduction}
\label{sec:intro}
In de Branges's theory of Hilbert spaces of entire functions, the
operator of multiplication by the independent variable plays a central
role.  From particular features of this operator, one can infer
properties of concrete realizations of de Branges spaces.  Conversely,
the particularities of a de Branges space determine the spectral properties of the
selfadjoint extensions of the corresponding multiplication
operator. This fact is specially useful when, via the so-called
functional model, one can identify Krein's entire operators and, more
generally, $n$-entire operators (in particular, regular and singular
Schr\"odinger operators as well as Jacobi operators) with the
multiplication operator in certain de Branges spaces
\cite{MR0012177,MR0011170,MR0011533,remling-1,IIII,MR3002855,us-6}.

This article uses singular perturbation theory for dealing with the
family of canonical selfadjoint extensions $S_\gamma$ of the
multiplication operator $S$ in a de Branges space $\cB$.  Our approach
to this issue is not reduced to the application of perturbation theory
to a concrete family of operators. Instead, we focus our attention to
functions in de Branges space theory that play a central role when
considering the operators $S_\gamma$ as a family of singular rank-one
perturbations of a certain selfadjoint extension of $S$.  To the best of our
knowledge, this way of dealing with the matter is new. We restrict our
considerations to the case when the operator $S$ is densely defined
since the other case has already been treated in \cite{III}. As a
matter of fact, this work can be regarded as a further development of
the results given in that paper. Singular perturbations are treated by
means of triplets of Hilbert spaces \cite[Chapter I]{berezanski}. We
combine this operator-theoretic approach with the properties of
functions in both the de Branges space and its set of associated
functions $\assoc\cB$ (see Section~\ref{sec:preliminaries}). In doing
so, we have two goals in mind. The first one consists in shedding
light on the properties of the linear spaces involved in the theory of
triplets of Hilbert spaces.  The second aim concerns the incorporation
of de Branges's theory into the theoretical framework of singular
rank-one perturbations. The results concerning the first goal can be
summarized as follows: Let $\cB_{+2}^{(\gamma)}$ and $\cF_{+1}$ be
$\dom(S_\gamma)$ and $\dom(S^*)$ with their respective graph
norms. Let $\cB_{-2}^{(\gamma)}$ and $\cF_{-1}$ be their respective
duals. By standard theory, one has
\begin{equation*}
\cB_{+2}^{(\gamma)}\subset\cF_{+1}
                   \subset\cB\subset\cF_{-1}\subset\cB_{-2}^{(\gamma)}\,.
\end{equation*}
We prove that $\cB_{+2}^{(\gamma)}$ is a de Branges space
(Theorem~\ref{thm:B+2-is-dB}) and the space $\cF_{+1}$ share many of
the properties of a de Branges space (Theorem~\ref{prop:F+1-is-dB}) 
but not all of them (Example~\ref{counter-example}).
We next prove that $\cF_{-1}$ is realized by (that is, is
isometrically isomorphic to) a de Branges space which is setwise equal
to $\assoc\cB$ (Proposition~\ref{prop:map-f-assoc} and
Theorem~\ref{thm:assoc-with-norm-is-dB}); we remark that $\cF_{-1}$ is
initially a Hilbert space of continuous linear functionals acting on
(the Hilbert space of entire functions) $\cF_{+1}$. Finally, we show
that $\cB_{-2}^{(\gamma)}$ is not a de Branges space but rather a
quotient space involving $\assoc\cB$
(Theorem~\ref{thm:b-minus-2-is-not-dB}).

On the subject of the second goal, our approach allows us to find
formulae for rendering the family of selfadjoint extensions of the
multiplication operator as a family of rank-one singular
perturbations. Namely (Theorems~\ref{teo:dom-selfadj-alternative} and
\ref{thm:action-s-gamma}),
\begin{equation*}
\dom(S_\gamma) =
\left\{\begin{gathered}
		g(z) = h(z) + b S_{\pi/2}(S_{\pi/2}+iI)^{-1}k(z,-i),
		\\
		h(z)\in\dom(S_{\pi/2}), b\in\C:\dual{s_0}{h}_2 =
		\pi b\left(\tan\gamma+\re\frac{s_{0}(i)}{s_{\pi/2}(i)}\right)
       \end{gathered}\right\},
\end{equation*}
where $k(z,w)$ is the reproducing kernel in $\cB$ and the functions
$s_\gamma(z)$ are given in (\ref{eq:functions-s}). Then the operator
$S_\gamma$ is the restriction of
\begin{equation*}
\tilde{S}_\gamma := \tilde{S}_{\pi/2}
				- \frac{\cot\gamma}{\pi}\dual{s_0}{\cdot}_{\cF}s_0(z)
\end{equation*}
to $\dom(S_\gamma)$; we note that $\tilde{S}_\gamma$ are maps from
$\dom(S^*)$ to $\cB_{-2}^{(\pi/2)}$ (the details are discussed in
Section~4).  Although, in an abstract setting, this kind of formulae
are known, here they are derived using function theoretical methods
pertaining to de Branges theory that make clear the prominent role
played by the functions $s_\gamma(z)$ in these formulae. Furthermore,
we obtain as by-products the Krein's formula for resolvents and some
objects related to it in terms of functions in de Branges's theory.

\section{Remarks on de Branges Hilbert spaces}
\label{sec:preliminaries}

Throughout this paper, inner products in Hilbert spaces are assumed
conjugate linear with respect to the left argument. We follow the
customary rule of denoting a function $f$ by its evaluation at an
arbitrary value of its argument $f(z)$. Also, we often denote the
action of an operator $B$ on a function $f(z)$ by $Bf(z)$ instead of
$(Bf)(z)$.

\medskip

A Hilbert space of entire functions $\cB$ is a de Branges (dB) space if it has
a reproducing kernel and is isometrically invariant under both the
conjugation $f(z)\mapsto f^\#(z):=\cc{f(\cc{z})}$  and the mapping
\begin{equation}
\label{eq:quotient-mapping}
f(z)\mapsto \frac{z-\cc{w}}{z-w}f(z)
\end{equation}
whenever $w\in\C$ is a non real zero of $f(z)$.

Alternatively, dB spaces can be defined in terms of functions of the
Hermite-Biehler class $\cHB$, that is, entire functions for which the
inequality $\abs{e(z)}>\abs{e(\cc{z})}$ holds for all $z\in\C_+$.
Indeed, given $e(z)\in\cHB$, one defines
\begin{equation*}
\cB(e) := \left\{f(z)\text{ entire}: \frac{f(z)}{e(z)} ,
\frac{f^\#(z)}{e(z)}\in\cH^2(\C_+)\right\},
\end{equation*}
where $\cH^2(\C_+)$ is the Hardy space on the upper half plane; the
inner product in $\cB(e)$ is given by
\begin{equation*}
\inner{f}{g} :=
\int_\R\frac{\cc{f(x)}g(x)}{\abs{e(x)}^2}dx.
\end{equation*}
According to \cite[Problem 50]{debranges68} $\cB(e)$ so defined is a dB
space with reproducing kernel
\begin{equation*}
k(z,w) = \begin{dcases}
			\frac{e^\#(z)e(\cc{w})-e(z)e^\#(\cc{w})}{2\pi i(z-\cc{w})},
			& w\ne\cc{z},
			\\
			\frac{{e^\#}'(z)e(z)-e'(z)e^\#(z)}{2\pi i},
			& w=\cc{z}.
		 \end{dcases}
\end{equation*}
On the other hand,
given a dB space $\cB\ne\{0\}$ there exists $e(z)\in\cHB$ such that
$\cB=\cB(e)$ isometrically \cite[Theorem 23]{debranges68}; such a function
however is not unique \cite[Theorem 1]{debranges60}.

A entire function $h(z)$ is associated to a given dB space $\cB$ if
\begin{equation*}
\frac{h(w)f(z)-h(z)f(w)}{z-w}\in\cB
\end{equation*}
for every $f(z)\in\cB$ and $w\in\C$ such that $h(w)\ne 0$.
The set of associated functions
is denoted $\assoc\cB$. By \cite[Lemma 4.5]{kaltenback} one has
$\assoc\cB = z\cB + \cB$.
Within $\assoc\cB$ lies the family of functions
\begin{equation}
\label{eq:functions-s}
s_\gamma(z) := \frac{i}{2}\left[e^{i\gamma}e(z) -
e^{-i\gamma}e^\#(z)\right],\quad\gamma\in[0,\pi).
\end{equation}
These functions are in bijective correspondence with the family of canonical
selfadjoint extension of the multiplication operator; see below. We note
that, in general, $s_\gamma(z)\in\assoc\cB\setminus\cB$ with the possible
exception of at most one of such functions \cite[Lemma 7]{debranges59}.
Extending the definition \eqref{eq:functions-s} to $\gamma\in\R$ one has the
identity
\begin{equation}
s_\gamma(z) = \cos(\gamma-\gamma_0) s_{\gamma_0}(z)
		     + \sin(\gamma-\gamma_0) s_{\gamma_0+\pi/2}(z),
\label{eq:s-beta}
\end{equation}
where $\gamma_0$ is fixed but otherwise arbitrary. Also,
\begin{equation}
k(z,w) = \begin{dcases}
         \frac{s_{\gamma_0+\pi/2}(z)s_{\gamma_0}(\cc{w})
         	- s_{\gamma_0+\pi/2}(\cc{w})s_{\gamma_0}(z)}{\pi(z-\cc{w})},
         & z\ne\cc{w},
         \\[1mm]
         \frac{s_{\gamma_0+\pi/2}'(z)s_{\gamma_0}(z)
         	- s_{\gamma_0+\pi/2}(z)s_{\gamma_0}'(z)}{\pi},
         & z=\cc{w}.
         \end{dcases}
\label{eq:k-given-by-s}
\end{equation}

The operator of multiplication by the independent variable is defined
by $(Sf)(z)=zf(z)$ with domain $\dom(S)$ maximal in $\cB$. This
operator is closed, completely nonselfadjoint and has deficiency
indices $(1,1)$. Additionally, $S$ is also regular, i.\,e. its
spectral core is empty, if and only if
$\cB=\cB(e)$ with $e(z)\in\cHB$ devoid of zeros in the real line.
Also, $\dom(S)$ may be either dense in $\cB$ or has codimension equal
to one; the latter happens if and only if there exists (a necessarily
unique) $\gamma\in[0,\pi)$ such that $s_\gamma(z)\in\cB$, in which
case $s_\gamma(z)$ is orthogonal to $\dom(S)$ \cite{debranges68}.

\medskip

From this point on, we consider only dB spaces with the property of
$S$ being densely defined; dB spaces with $\dom(S)$ having non zero
codimension have been discussed in \cite{III}.

\medskip

Observe that
\begin{equation*}
\inner{(S^*-\cc{w})k(\cdot,w)}{g(\cdot)}
	= \inner{k(\cdot,w)}{(S - w)g(\cdot)}
	= 0
\end{equation*}
for all $g(z)\in\dom(S)$. Therefore
\begin{equation*}
k(z,w)\in\ker(S^*-\cc{w}I),\quad w\in\C.
\end{equation*}
Thus, since we assume $S$ densely defined, $S^*$ can be described as
\begin{subequations}
\label{eq:adjoint-S}
\begin{gather}
\dom(S^*) = \left\{\begin{gathered}
				   g(z) = h(z) + a^{(+)}k(z,\cc{w}) + a^{(-)}k(z,w):
				   \\[1mm]
				   h(z)\in\dom(S),\; a^{(\pm)}\!\in\C,\; \im(w)>0
                   \end{gathered}\right\},
\label{eq:adjoint-S-dom}
\\[1mm]
S^*g(z) = zh(z) + w\,a^{(+)}k(z,\cc{w}) + \cc{w}\,a^{(-)}k(z,w).
\label{eq:adjoint-S-action}
\end{gather}
\end{subequations}

The canonical selfadjoint extensions of $S$ are the selfadjoint
restrictions of $S^*$. A standard description of them is given in
\cite{kaltenback}, where the connection to the family of functions
$s_\gamma(z)$ is made explicit:
\begin{subequations}
\label{eq:definition-selfadj}
\begin{gather}
\dom(S_\gamma) = \left\{g(z)=\frac{f(z) -
\frac{s_\gamma(z)}{s_\gamma(w)}f(w)}{z-w}:
				f(z)\in\cB,\; w\in\C\setminus\R\right\},
	\label{eq:dom-selfadj}
	\\[2mm]
S_\gamma g(z) = zg(z) + f(w)\frac{s_\gamma(z)}{s_\gamma(w)},\quad \gamma\in[0,\pi).
	\label{eq:action-selfaj}
\end{gather}
\end{subequations}

It follows from \eqref{eq:definition-selfadj} that
$\spec(S_\gamma)=\{\text{zeros of }s_\gamma(z)\}$ and the corresponding
eigenfunctions are of the form $s_\gamma(z)/(z-\lambda)$,
$\lambda\in\spec(S_\gamma)$. Note that \eqref{eq:definition-selfadj}
also implies
\begin{equation*}
(S_\gamma - wI)^{-1}f(z) = g(z),\quad w\not\in\spec(S_\gamma).
\end{equation*}
If $S$ is also regular, then $\bigcap_\gamma\spec(S_\gamma)=\emptyset$,
$\bigcup_\gamma\spec(S_\gamma)=\R$, and $\spec(S_\gamma)$ and
$\spec(S_{\gamma'})$ interlace whenever $\gamma\ne\gamma'$.

\begin{lemma}
\label{lem:quotient-diff-reprod-kernel}
Suppose $v\not\in\spec(S_\gamma)$, $w\not\in\spec(S_\gamma)$ and $v\ne w$. Then,
\begin{equation*}
\frac{k(z,\cc{w})-k(z,\cc{v})}{w-v}
	= (S_{\gamma}-vI)^{-1}k(z,\cc{w})
	  + \frac{1}{w-v}\left(\frac{s_{\gamma}(w)}{s_{\gamma}(v)}-1\right)
	    k(z,\cc{v}).
\end{equation*}
\end{lemma}
\begin{proof}
For every $g(z)\in\cB$ one has
\begin{align*}
\inner{k(\cdot,\cc{w})}{(S_{\gamma}-\cc{v}I)^{-1}g(\cdot)}
	&= \frac{g(\cc{w})-\frac{s_{\gamma}(\cc{w})}{s_{\gamma}(\cc{v})}g(\cc{v})}
			{\cc{w}-\cc{v}}
	\\[1mm]
	&= \frac{s_{\gamma}(\cc{w})}{s_{\gamma}(\cc{v})}
	   \frac{g(\cc{v})-\frac{s_{\gamma}(\cc{v})}{s_{\gamma}(\cc{w})}g(\cc{w})}
			{\cc{v}-\cc{w}}
	\\[2mm]
	&= \frac{s_{\gamma}(\cc{w})}{s_{\gamma}(\cc{v})}
	   \inner{k(\cdot,\cc{v})}{(S_{\gamma}-\cc{w}I)^{-1}g(\cdot)}
\end{align*}
so
\begin{equation}
\label{eq:symmetry}
s_{\gamma}(v)(S_{\gamma}-vI)^{-1}k(z,\cc{w})
	= s_{\gamma}(w)(S_{\gamma}-wI)^{-1}k(z,\cc{v}).
\end{equation}
Then
\begin{align*}
k(z,\cc{w})-k(z,\cc{v})
	&= \left[\frac{s_{\gamma}(w)}{s_{\gamma}(v)}(S_{\gamma}-vI)
	         (S_{\gamma}-wI)^{-1} - I\right]k(z,\cc{v})
	\\[1mm]
	&= \left[\left(\frac{s_{\gamma}(w)}{s_{\gamma}(v)}-1\right)I
	    +(w-v)\frac{s_{\gamma}(w)}{s_{\gamma}(v)}(S_{\gamma}-wI)^{-1}\right]
	    k(z,\cc{v}),
\end{align*}
yielding the desired result after one more use of \eqref{eq:symmetry}.
\end{proof}

Another consequence of \eqref{eq:symmetry} is
\begin{equation}
\label{eq:cayley-transf}
s_\gamma(\cc{w})k(z,\cc{w})
	= s_\gamma(w)U(w) k(z,w),
\end{equation}
where
\begin{equation}
  \label{eq:cayley-transform}
  U(w):= (S_\gamma-\cc{w}I)(S_\gamma-wI)^{-1}=
  I +(w-\cc{w})(S_\gamma-wI)^{-1},\quad w\in\C\setminus\R,
\end{equation}
is the Cayley transform.
This in turn yields another characterization of the adjoint operator.

\begin{lemma}
Given $\gamma\in[0,\pi)$ and $w\in\C:\im(w)>0$, the adjoint operator $S^*$ can
be described as follows:
\begin{subequations}
\label{eq:adjoint-S-alt}
\begin{gather}
\dom(S^*) = \left\{\begin{gathered}
		     f(z) = g(z) + b (S_\gamma-(\re w)I)(S_\gamma - wI)^{-1}k(z,w):
		     \\[1mm]
		     g(z)\in\dom(S_\gamma),\; b\in\C
             \end{gathered}\right\},
\label{eq:adjoint-S-alt-dom}
\\[1mm]
S^*f(z) = S_\gamma g(z)
	      + b ((\re w)S_\gamma - \abs{w}^2I)(S_\gamma - wI)^{-1}k(z,w).
\label{eq:adjoint-S-alt-action}
\end{gather}
\end{subequations}
\end{lemma}
\begin{proof}
Define $l(z,w):=k(z,w)/s_\gamma(\cc{w})$. Every $f(z)\in\dom(S^*)$
is of the form
\begin{equation*}
f(z) = h(z) + a^{(+)}l(z,\cc{w}) + a^{(-)}l(z,w),\quad h(z)\in\dom(S),\quad
	a^{(\pm)}\in\C.
\end{equation*}
With the help of \eqref{eq:cayley-transf}, this can be written as
\begin{equation*}
f(z) = g(z) + \frac{a^{(+)}+a^{(-)}}{2}\left(U(w) + I\right)l(z,w),
\end{equation*}
where
\begin{equation*}
g(z) = h(z) + \frac{a^{(+)}-a^{(-)}}{2}\left(U(w) - I\right)l(z,w)
\end{equation*}
belongs to the domain of $S_\gamma$ (see
(\ref{eq:cayley-transform})). Since,
$U(w) + I=2(S_\gamma-(\re w)I)$, which can be verified by a
straightforward computation,
\eqref{eq:adjoint-S-alt-dom} has been established.

To prove \eqref{eq:adjoint-S-alt-action}, it suffices to note that
\begin{align*}
S^*f(z) &= S^*g(z) + \frac{a^{(+)}+a^{(-)}}{2}
			\left(S^*l(z,\cc{w}) + S^*l(z,w)\right)
		\\[1mm]
		&= S_\gamma g(z) + \frac{a^{(+)}+a^{(-)}}{2}
			\left(wU(w) + \cc{w}I\right)l(z,w)
\end{align*}
and rewrite $wU(w) + \cc{w}I$ using (\ref{eq:cayley-transform}).
\end{proof}

The description \eqref{eq:adjoint-S-alt} of $S^*$ takes its simplest form
when $w=i$ (or $w=-i$). In such a case, one obtains (cf. \cite{albeverio-kurasov})
\begin{subequations}
\label{eq:adjoint-S-usual}
\begin{gather}
\dom(S^*) = \left\{\begin{gathered}
		     f(z) = g(z) + b S_\gamma(S_\gamma - iI)^{-1}k(z,i):
		     \\[1mm]
		     g(z)\in\dom(S_\gamma),\; b\in\C
             \end{gathered}\right\},
\label{eq:adjoint-S-usual-dom}
\\[1mm]
S^*f(z) = S_\gamma g(z)
	      - b (S_\gamma - iI)^{-1}k(z,i).
\label{eq:adjoint-S-usual-action}
\end{gather}
\end{subequations}

\begin{remark}
\label{rem:HB-no-real-zeros}
From this point on, de Branges spaces are assumed to be generated by
Hermite-Biehler functions free of zeros on the real line.
Equivalently, de Branges spaces are supposed to have the property
that, for every $x\in\R$, there exists $f(z)$ in the space such that
$f(x)\ne 0$. This assumption, which means that the operator of
multiplication is regular, entails no essential loss of generality.
\end{remark}


\section{Scales of dB spaces}
\label{sec:scale-db-spaces}

Most of the following discussion is based on
\cite{albeverio-kurasov97a,albeverio-kurasov97,albeverio-kurasov} with
some slight modifications. Other related references are
\cite{hassi-desnoo97,hassi-kaltenback-desnoo,hassi-desnoo98} and the
classical book by Berezanski{\u\i} \cite[Chapter I]{berezanski}.

\medskip

Given $\gamma\in[0,\pi)$, let $\cB_{+2}^{(\gamma)}:=\dom(S_\gamma)$ equipped
with the graph inner product
\begin{equation*}
\inner{f}{g}_{+2}
	:= \inner{S_\gamma f}{S_\gamma g} + \inner{f}{g}
	 = \inner{(S_\gamma - iI)f}{(S_\gamma - iI)g}.
\end{equation*}
Due to $S_\gamma$ being closed, $\cB_{+2}^{(\gamma)}$ is a Hilbert space.
Its dual $\cB_{-2}^{(\gamma)}$ is the completion of $\cB$ under
the norm induced by the inner product
\begin{equation*}
\inner{f}{g}_{-2}
	:= \inner{f}{(S_\gamma^2 + I)^{-1} g}
	 = \inner{(S_\gamma - iI)^{-1}f}{(S_\gamma - iI)^{-1}g}.
\end{equation*}
Since
\begin{equation*}
\norm{S_\gamma f}_{-2}\le \norm{f}, \quad f(z)\in\dom(S_\gamma),
\end{equation*}
the operator $S_\gamma$ has a unique contractive continuation
$\hat{S_\gamma}:\cB\to\cB_{-2}^{(\gamma)}$. Similarly,
\begin{equation*}
R_\gamma(w) :=(S_\gamma - wI)^{-1}:\cB\to\cB_{+2}^{(\gamma)},\quad
w\in\C\setminus\spec(S_\gamma),
\end{equation*}
has a unique bounded extension $\hat{R}_\gamma(w):\cB_{-2}^{(\gamma)}\to\cB$.
The latter operators satisfy the resolvent identity
\begin{equation}
\label{eq:ext-res-identity}
\hat{R}_\gamma(w) - \hat{R}_\gamma(v)
	= (w-v)R_\gamma(w)\hat{R}_\gamma(v),
\end{equation}
for $w\in\C\setminus\spec(S_\gamma)$ and $v\in\C\setminus\spec(S_\gamma)$.
In terms of $\hat{R}_\gamma(w)$, the associated pairing
$\dual{\cdot}{\cdot}_2:\cB_{-2}^{(\gamma)}\times\cB_{+2}^{(\gamma)}\to\C$
can be written as
\begin{equation}
  \label{eq-2pairing-expression}
\dual{\varphi}{g}_{2} =
	\inner{\hat{R}_\gamma(-i)\varphi}{(S_\gamma-iI)g},\quad
	\varphi\in\cB_{-2}^{(\gamma)},\quad g(z)\in\cB_{+2}^{(\gamma)}.
\end{equation}

Before going into our first result, let us note that
\begin{equation}
\label{eq:norm+2}
\norm{g}_{+2}
= \norm{(S_{\gamma}-iI)g}
 = \norm{f}\,,
\end{equation}
when $g(z)$ and $f(z)$ are related as in \eqref{eq:dom-selfadj}.

\begin{theorem}
  \label{thm:B+2-is-dB}
$\cB_{+2}^{(\gamma)}$ is a dB space.
\end{theorem}
\begin{proof}
	The linear functional of point evaluation is continuous in
	$\cB_{+2}^{(\gamma)}$. Indeed, given $g(z)\in\cB_{+2}^{(\gamma)}$
	there exists $f(z)\in\cB$ such that
	\begin{equation*}
	g(z) = (S_\gamma - iI)^{-1}f(z).
	\end{equation*}
	Let $k(w,z)$ be the reproducing kernel in $\cB$. Then,
	\begin{equation*}
	g(z) = \inner{k(\cdot,z)}{g(\cdot)}
	     = \inner{k(\cdot,z)}{(S_\gamma - iI)^{-1}f(\cdot)}
	     = \inner{(S_\gamma + iI)^{-1}k(\cdot,z)}{f(\cdot)}
	\end{equation*}
	so
	\begin{equation*}
	\abs{g(z)} \le \norm{(S_\gamma +iI)^{-1}}\norm{k(\cdot,z)}\norm{f}
	\end{equation*}
	which along with \eqref{eq:norm+2} imply the assertion. Moreover, it
	follows that
	\begin{equation*}
	\abs{g(z)-h(z)}^2 \le C k(z,z)\norm{g - h}_{+2}^2
	\end{equation*}
	so point evaluation is uniformly continuous over compact subsets
	of $\C$.

	Next, let us verify that the conjugation $^\#$ is an isometry in
	$\cB_{+2}^{(\gamma)}$. Consider
	\begin{equation*}
	g(z) = (S_\gamma + iI)^{-1}f(z)
	     = \frac{f(z)-\frac{s_\gamma(z)}{s_\gamma(-i)}f(-i)}{z+i}
	\end{equation*}
	with $f(z)\in\cB$. We have
	\begin{equation*}
	g^\#(z) =
	\frac{f^\#(z)-\frac{s_\gamma(z)}{\cc{s_\gamma(-i)}}\cc{f(-i)}}{z-i}
	        = \frac{f^\#(z)-\frac{s_\gamma(z)}{s_\gamma(i)}f^\#(i)}{z-i}
	        = (S_\gamma - iI)^{-1}f^\#(z),
	\end{equation*}
	thus $g^\#(z)\in\cB_{+2}^{(\gamma)}$ since $f^\#(z)\in\cB$. Moreover,
	\begin{align*}
	 \norm{g^\#}_{+2}
	   = \norm{(S_\gamma - iI)g^\#}
	  &= \norm{f^\#}
	  \\[1mm]
	  &= \norm{f}
	   = \norm{(S_\gamma + iI)g}
	   = \norm{g}_{+2},
	\end{align*}
	proving the assertion; the fact that the Cayley transform is unitary has
	been used along the way.

	Now, suppose that $v_0$ is a non-real zero of
	$g(z)\in\cB_{+2}^{(\gamma)}$ and consider $f(z)\in\cB$ such that
	\begin{equation*}
	g(z) = (S_\gamma -\cc{v_0}I)^{-1}f(z)
	     = \frac{f(z)-\frac{s_\gamma(z)}{s_\gamma(\cc{v_0})}f(\cc{v_0})}
	            {z-\cc{v_0}}\,.
	\end{equation*}
	Thus, the assumption $g(v_0)=0$ yields
	\begin{equation*}
	f(v_0) = \frac{s_\gamma(v_0)}{s_\gamma(\cc{v_0})}f(\cc{v_0}).
	\end{equation*}
	Hence
	\begin{equation*}
	\frac{z-\cc{v_0}}{z-v_0}g(z)
		= \frac{f(z)-\frac{s_\gamma(z)}{s_\gamma(\cc{v_0})}f(\cc{v_0})}{z-v_0}
		= \frac{f(z)-\frac{s_\gamma(z)}{s_\gamma(v_0)}f(v_0)}{z-v_0}
		= (S_\gamma - v_0I)^{-1}f(z)
	\end{equation*}
	so $(z-\cc{v_0})(z-v_0)^{-1}g(z)\in\cB_{+2}^{(\gamma)}$. Finally,
	\begin{align*}
	\norm{\frac{(\cdot)-\cc{v_0}}{(\cdot)-v_0}g(\cdot)}_{+2}
	&= \norm{(S_\gamma + iI)(S_\gamma - \cc{v_0}I)^{-1}
		    (S_\gamma - \cc{v_0}I)(S_\gamma -v_0I)^{-1}f}
	\\[1mm]
	&= \norm{(S_\gamma + iI)(S_\gamma -\cc{v_0}I)^{-1}f}
	 = \norm{g}_{+2},
	\end{align*}
	where again the isometric character of the Cayley transform has been used.
	The proof is now complete.
\end{proof}

\begin{remark}
Denote by $k_{+2}(z,w)$ the reproducing kernel in $\cB_{+2}^{(\gamma)}$. Since
\begin{align*}
g(w) &= \inner{k(\cdot,w)}{g(\cdot)}
	    \\[1mm]
	 &= \inner{(S_\gamma + iI)(S_\gamma - iI)^{-1}k(\cdot,w)}
	          {(S_\gamma + iI)(S_\gamma - iI)^{-1}g(\cdot)}
	    \\[1mm]
	 &= \inner{(S_\gamma + iI)(S_\gamma^2 + I)^{-1}k(\cdot,w)}
	          {(S_\gamma + iI)g(\cdot)}
	 \\[1mm]
	 &= \inner{(S_\gamma^2 + I)^{-1}k(\cdot,w)}{g(\cdot)}_{+2}
\end{align*}
for all $g(z)\in\cB_{+2}^{(\gamma)}$ and $w\in\C$, it follows that
\begin{equation*}
k_{+2}(z,w) = (S_\gamma^2 + I)^{-1}k(z,w).
\end{equation*}
\end{remark}

\medskip

Let us introduce a second scale of spaces associated with $\cB$ and $S$.
Define $\cF_{+1}:= \dom(S^*)$ equipped with the graph inner product
\begin{equation*}
\inner{f}{g}_{+\cF}
	:= \inner{S^* f}{S^* g} + \inner{f}{g}.
\end{equation*}
Due to \eqref{eq:adjoint-S-usual-dom}, for $f(z)$ and $g(z)$ in
$\dom(S^*)$, one has
\begin{gather*}
f(z) = h(z) + b S_\gamma(S_\gamma-iI)^{-1}k(z,i),
	\quad h(z)\in\dom(S_\gamma),\quad b\in\C,
	\\[1mm]
g(z) = n(z) + d S_\gamma(S_\gamma-iI)^{-1}k(z,i),
	\quad n(z)\in\dom(S_\gamma),\quad d\in\C.
\end{gather*}
From these expansions, using \eqref{eq:adjoint-S-usual-action}, one obtains
\begin{equation}
\label{eq:inner-product-F+}
\inner{f}{g}_{+\cF}
	= \inner{h}{n}_{+2}
        + \cc{b}d\norm{(S_\gamma-i)^{-1}k(\cdot,i)}_{+2}^2=
        \inner{h}{n}_{+2}
        + \cc{b}d\norm{k(\cdot,i)}^2.
\end{equation}
This clearly implies $\norm{f}_{+\cF}=\norm{f}_{+2}$ whenever
$f\in\dom(S_\gamma)$, which gives rise to the scale
\begin{equation*}
\cB_{+2}^{(\gamma)}\subset\cF_{+1}
                   \subset\cB\subset\cF_{-1}\subset\cB_{-2}^{(\gamma)};
\end{equation*}
here $\cF_{-1}$ is the completion of $\cB$ with respect
to the norm
\begin{equation*}
\norm{f}_{-\cF}
:= \sup_{g(z)\in\dom(S^*)\setminus\{0\}}
	\frac{\abs{\inner{f}{g}}}{\norm{g}_{+\cF}}.
\end{equation*}

\begin{remark}
$\cB_{+2}^{(\gamma)}$ is not dense in $\cF_{+1}$, a fact
obvious upon inspection of \eqref{eq:inner-product-F+}.
\end{remark}

For the next result we will use the following consequence of
Lemma~\ref{lem:quotient-diff-reprod-kernel},
\begin{align}
U(\cc{w_0})k(v,z)
	&= (S^*-w_0I)(S_\gamma-\cc{w_0}I)^{-1}k(v,z)\nonumber
	\\[1mm]
	&= \frac{\cc{z}-w_0}{\cc{z}-\cc{w_0}}k(v,z)
	   - \frac{\cc{w_0}-w_0}{\cc{z}-\cc{w_0}}
          \frac{s_\gamma(\cc{z})}{s_\gamma(\cc{w_0})}k(v,w_0),
          \quad z\ne w_0,
	   \label{eq:cayley-transform-on-k}
\end{align}
where $U(w)$ is the Cayley transform \eqref{eq:cayley-transform}.

\begin{theorem}
\label{prop:F+1-is-dB}
$\cF_{+1}$ is a Hilbert space of entire functions with reproducing
kernel, isometrically invariant under the conjugation
$f(z)\mapsto f^\#(z):=\cc{f(\cc{z})}$, and invariant under the mapping
$f(z)\mapsto \frac{z-\cc{w}}{z-w}f(z)$ whenever $w\in\C$ is a non real
zero of $f(z)$.
\end{theorem}
\begin{proof}
Given $f(z), g(z)\in\dom(S^*)$ and $w\in\C$,
\begin{equation*}
\abs{g(w)-f(w)}
	= \abs{\inner{k(\cdot,w)}{g(\cdot)-f(\cdot)}}
	\le \norm{k(\cdot,w)}_{-\cF}\norm{g-f}_{+\cF},
\end{equation*}
which implies that point evaluation is continuous in $\cF_{+1}$.

Next, suppose $f(z)\in\dom(S^*)$. Then, according to \eqref{eq:adjoint-S-dom},
\begin{equation*}
  f(z) = h(z)+a^{(+)}k(z,\cc{w})+ a^{(-)}k(z,w),
\end{equation*}
for some $h(z)\in\dom(S)$, $a,b\in\C$ and some fixed $w\in\C_+$. Thus
\begin{equation*}
f^\#(z) = h^\#(z)+\cc{a^{(+)}k(\cc{z},\cc{w})}+\cc{a^{(-)}k(\cc{z},w)}.
\end{equation*}
From this equality, one concludes that $f^\#(z)\in\dom(S^*)$ since $\cc{k(\cc{z},w)}=k(z,\cc{w})$ and $h^\#(z)\in\dom(S)$. Moreover, recalling that $S^*$ commutes
with the conjugation $^\#$, one obtains
\begin{equation*}
\norms{g^\#}_{+\cF}^2
	= \norm{S^*g^\#}^2 + \norms{g^\#}^2
	= \norm{(S^*g)^\#}^2 + \norms{g}^2
	= \norms{g}_{+\cF}^2.
\end{equation*}

Finally, suppose that $f(z)\in\dom(S^*)$ has a non real zero $w_0$. Define
$g(z):=U(w_0)f(z)$, where $U(w)$ is given in \eqref{eq:cayley-transform}. Since $U(w)$
maps $\dom(S^*)$ into itself, the function $g(z)$ belongs to $\dom(S^*)$. Moreover,
\begin{equation*}
\inner{k(\cdot,z)}{g(\cdot)}
	= \inner{U(\cc{w_0})k(\cdot,z)}{f(\cdot)}
	= \frac{z-\cc{w_0}}{z-w_0}f(z),
\end{equation*}
as implied by \eqref{eq:cayley-transform-on-k}.
\end{proof}

In general, $\cF_{+1}$ is not a dB space. This is somewhat hinted in the 
proof above: If $f(z)\in\dom(S^*)$ has a non real zero $w_0$, then
$(z-\cc{w_0})(z-w_0)^{-1}f(z)\in\dom(S^*)$ but 
$\norm{(\cdot-\cc{w_0})(\cdot-w_0)^{-1}f}_{+\cF}$ does not necessarily equal 
$\norm{f}_{+\cF}$. The next (counter) example illustrates this fact.

\begin{example}
\label{counter-example}
Consider the Paley-Wiener space
\begin{equation*}
\cPW_a := \left\{f(z)=\int_{-a}^a e^{izx}\varphi(x)dx : \varphi\in L_2(-a,a)\right\},
\end{equation*}
whose norm clearly obeys $\norm{f}_{\cPW_a} = \norm{\varphi}_{L_2(-a,a)}$.
Moreover, the Fourier transform is a unitary mapping between the maximally 
defined operator $\varphi\mapsto i\varphi'$ in $L_2(-a,a)$ and $S^*$, thus
\[
\dom(S^*) 
= \left\{f(z)=\int_{-a}^a e^{izx}\varphi(x)dx : \varphi\in \text{AC}[-a,a]\right\}.
\]
It follows that
\[
\norm{f}_{+\cF}^2 = \norm{\varphi'}_{L_2(-a,a)}^2 + \norm{\varphi}_{L_2(-a,a)}^2.
\]
Its selfadjoint restrictions $S_\gamma$, $\gamma\in [0,\pi)$, have domains
\[
\dom(S_\gamma) 
= \left\{f(z)=\int_{-a}^a e^{izx}\varphi(x)dx : \varphi\in \text{AC}[-a,a]
	\text{ and } \varphi(a)=e^{2i\gamma}\varphi(-a)\right\}.
\]
Suppose $f(z) = \dom(S^*)$ has a non real zero $w_0$. If $f(z)$ is the Fourier
transform of $\varphi\in L_2(-a,a)$, one has
\[
\frac{z-\cc{w_0}}{z- w_0}f(z)
	= \int_{-a}^a e^{izx} \eta(x) dx,
\]
where (choosing $\gamma=0$)
\[
\eta(x) 
	= (A_0 - \cc{w_0}I)(A_0 - w_0 I)^{-1}\varphi(x)
	= \varphi(x) - 2i\re(w_0) \int_{-a}^x e^{iw_0 (y-x)}\varphi(y)dy.
\]
Now set $\varphi(x)=e^{-x}$, whose image is
\[
f(z) = \begin{dcases}
		\frac{2\sin(z+i)a}{z+i},& z\ne -i,
		\\
		2a                     ,& z=   -i,	
	   \end{dcases}
\]
and choose $w_0 = \pi/a - i$ so
\[
\eta(x) = e^{-x} + i\frac{2a}{\pi}e^{-x}\left(1+ e^{-i\pi x/a}\right).
\]
One can readily verify that $\norm{\varphi'}_{L_2(-a,a)}\ne\norm{\eta'}_{L_2(-a,a)}$.
\end{example}

Denote by $\dual{\cdot}{\cdot}_\cF$ the pairing between $\cF_{-1}$ and
$\cF_{+1}$. The expression $\dual{f}{\phi}_\cF$ for $f\in\cF_+$ and
$\phi\in\cF_{-1}$ of course means $\cc{\dual{\phi}{f}}_\cF$.
\begin{proposition}
\label{prop:map-f-assoc}
The mapping $\dual{k(\cdot, z)}{\cdot}_\cF$ is a bijection between
$\cF_{-1}$ and $\assoc\cB$.
\end{proposition}
\begin{proof}
Consider $\psi\in\cF_{-1}$ and define
\begin{equation}
\label{eq:F-minus-one-as-assoc}
f(z):=\dual{k(\cdot,z)}{\psi}_\cF.
\end{equation}
Given $w\in\C$, choose $\gamma$ such that $w\not\in\spec(S_\gamma)$. For
any $g(z)\in\cB$,
\begin{align*}
\frac{f(z)g(w)-f(w)g(z)}{z-w}
	&= \frac{1}{z-w}\dual{k(\cdot,z)}{g(w)\psi-f(w)g}_\cF
	\\[1mm]
	&= \dual{\frac{k(\cdot,z)-k(\cdot,w)}{\cc{z}-\cc{w}}}{g(w)\psi-f(w)g}_\cF
	\\[1mm]
	&= \dual{R_\gamma(w)k(\cdot,z)}{g(w)\psi-f(w)g}_\cF
	\\[4mm]
	&= \inner{k(\cdot,z)}{n(\cdot,w)},
\end{align*}
where
$n(z,w):=g(w)\tilde{R}_\gamma(w)\psi(z)-f(w)R_\gamma(w)g(z)$;
Lemma~\ref{lem:quotient-diff-reprod-kernel} has been used in the computation
above. Since $n(z,w)\in\cB$ (as a function of $z$), it follows that
$f(z)\in\assoc\cB$. Therefore, $\dual{k(\cdot,z)}{\cdot}_\cF$ maps
$\cF_{-1}$ into $\assoc\cB$.

Suppose now $f(z)\in\assoc\cB$, that is, $f(z)=zg(z)+h(z)$ for some
$g(z),h(z)\in\cB$.
Consider $\{g_n(z)\}_{n\in\N}\subset\dom(S)$ such that
$\norm{g-g_n}\to 0$ and define $f_n(z):=zg_n(z)+h(z)$. Since
\begin{equation*}
\frac{\abs{\inner{f_n-f_m}{l}}}{\norm{l}_{+\cF}}
	=  \frac{\abs{\inner{g_n-g_m}{S^*l}}}
	        {\norm{l}_{+\cF}}
	\le\norm{g_n-g_m}
\end{equation*}
for any $l(z)\in\dom(S^*)$, it follows that $\{f_n(z)\}_{n\in\N}$ is
Cauchy convergent in $\cF_{-1}$.  Let $\phi\in\cF_{-1}$ be its limit,
which is well defined (i.\,e., independent of the sequence) by
standard arguments. Since $k(z,w)\in\dom(S^*)$ for all $w\in\C$, we
can define $\tilde{f}(z):=\dual{k(\cdot,z)}{\phi}_{\cF}$.  Now, on one
hand we have
\begin{equation}
\label{eq:convergent-compact-1}
\abss{\tilde{f}(z)-f_n(z)}
	=  \abss{\dual{k(\cdot,z)}{\phi-f_n}_{\cF}}
       \le\norms{k(\cdot,z)}_{+\cF}\norms{\phi-f_n}_{-\cF}.
\end{equation}
On the other hand, we have $\assoc\cB=\cB(e_1)$ where we can choose
$e_1(z)=(z+i)e(z)$. Let $k_1(z,w)$ be the reproducing kernel in
$\cB(e_1)$ \cite{langer-woracek}.
Then,
\begin{align}
\abs{f(z)-f_n(z)}^2
	&\le \norm{k_1(\cdot,z)}_{\cB(e_1)}^2\norm{f-f_n}_{\cB(e_1)}^2
	\nonumber
	\\[1mm]
	& =  k_1(z,z)\int_\R\frac{\abs{xg(x)-xg_n(x)}^2}{(x^2+1)\abs{e(x)}^2}dx
	\nonumber
	\\[1mm]
	&\le k_1(z,z)\norm{g-g_n}^2
	\label{eq:convergent-compact-2}.
\end{align}
Therefore,
\begin{equation*}
\abss{\tilde{f}(z)-f(z)}
	\le \abss{\tilde{f}(z)-f_n(z)} + \abss{f(z)-f_n(z)}
\end{equation*}
can be made arbitrarily small in compact subsets of $\C$ due to
\eqref{eq:convergent-compact-1} and \eqref{eq:convergent-compact-2}.
This in fact shows that $\dual{k(\cdot, z)}{\cdot}_\cF:\cF_{-1}\to\assoc\cB$
is onto. Since $\dual{k(\cdot,z)}{\phi}_\cF\equiv 0$, one has $\phi=0$.
Thus this mapping is also injective.
\end{proof}

On $\assoc\cB$ ---and with some abuse of notation--- define the norm
\begin{equation}
\label{eq:norm-for-assoc}
\norm{f}_{-\cF}:=\norm{\psi}_{-\cF},
\end{equation}
with $f(z)$ and $\psi$ related as in \eqref{eq:F-minus-one-as-assoc}.
This makes
$\dual{k(\cdot, z)}{\cdot}_\cF:\cF_{-1}\to\assoc\cB$ an isometry and,
taking into account the underlying inner product, $\assoc\cB$ is a Hilbert space.

\begin{remark}
In the terminology of \cite{woracek}, $\assoc\cB$ is just a dB-normable linear
space of entire functions. It is well known that, if $\cB = \cB(e(z))$ for 
some $e(z)\in\cHB$, then $\assoc\cB = \cB((z+w)e(z))$ setwise, for any $w\in\C_+$.
The next theorem shows that \eqref{eq:norm-for-assoc} provides a different norm
under which $\assoc\cB$ becomes a dB space.
\end{remark}

\begin{theorem}
  \label{thm:assoc-with-norm-is-dB}
The Hilbert space $\assoc\cB$ given above, i.\,e. with the norm
\eqref{eq:norm-for-assoc}, is
a dB space.
\end{theorem}
\begin{proof}
  Consider $f(z),g(z)\in\assoc\cB$. In accordance with
  \eqref{eq:F-minus-one-as-assoc}, they are images of some
  $\psi,\eta\in\cF_{-1}$. Given $w\in\C$,
\begin{equation*}
\abs{f(w)-g(w)}
	=   \abs{\dual{k(\cdot,w)}{\psi-\eta}_\cF}
	\le \norm{k(\cdot,w)}_{+\cF}\norm{\psi-\eta}_{-\cF}.
\end{equation*}
Then,  the functional of point evaluation is continuous.

Let $^\#:\cF_{-1}\to\cF_{-1}$ be the extension of the conjugation
$f(z)\mapsto f^\#(z)=\cc{f(\cc{z})}$ defined in $\cB$, given by the rule
\begin{equation*}
\dual{\psi^\#}{f}_{\cF}:=\cc{\dual{\psi}{f^\#}_{\cF}},\quad
	\psi\in\cF_{-1},\quad f(z)\in\dom(S^*).
\end{equation*}
Since $\assoc\cB$ coincides with $\cB(e_1)$ setwise
\cite{langer-woracek}, if $f(z)\in\assoc\cB$, then
$f^\#(z)\in\assoc\cB$. Indeed, if
\begin{equation*}
f(z) = \dual{k(\cdot,z)}{\psi}_\cF\quad \text{then}\quad
f^\#(z) = \dual{k(\cdot,z)}{\psi^\#}_\cF;
\end{equation*}
this follows from the fact that $\cc{k(\cc{z},w)}=k(z,\cc{w})$. Moreover,
\begin{align*}
\norm{f^\#}_{-\cF}
     = \norm{\psi^\#}_{-\cF}
	&= \sup_{g(z)\in\cF_{+1}\setminus\{0\}}
	   \frac{\abs{\dual{\psi^\#}{g}_\cF}}{\norm{g}_{+\cF}}
	\\[1mm]
	&= \sup_{g(z)\in\cF_{+1}\setminus\{0\}}
	   \frac{\abs{\dual{\psi}{g^\#}_\cF}}{\norm{g^\#}_{+\cF}}
	 = \norm{\psi}_{-\cF} = \norms{f}_{-\cF}.
\end{align*}
Finally, let us consider the mapping \eqref{eq:quotient-mapping}. Suppose
$f(z)=\dual{k(\cdot,z)}{\psi}_\cF\in\assoc\cB$ such that $f(w_0)=0$
for some $w_0\in\C\setminus\R$. Let $\hat{U}(w):\cF_{-1}\to\cF_{-1}$ denote
the dual of the Cayley transform $U(\cc{w}):\cF_{+1}\to\cF_{+1}$. Due to
\eqref{eq:cayley-transform-on-k},
\begin{equation*}
\dual{k(\cdot,z)}{\hat{U}(w_0)\psi}_\cF
	= \dual{U(\cc{w_0})k(\cdot,z)}{\psi}_\cF
    = \frac{z-\cc{w_0}}{z-w_0}f(z).
\end{equation*}
Furthermore,
\begin{align*}
\norm{\hat{U}(\cc{w_0})\psi}_{-\cF}
	= \sup_{g(z)\in\cF_{+1}\setminus\{0\}}
	   \frac{\abs{\dual{\hat{U}(\cc{w_0})\psi}{g}_\cF}}{\norm{g}_{+\cF}}
	&= \sup_{g(z)\in\cF_{+1}\setminus\{0\}}
	   \frac{\abs{\dual{\psi}{U(w_0)g}_\cF}}
	        {\norm{U(\cc{w_0})U(w_0)g}_{+\cF}}
	\\[1mm]
	&= \sup_{g(z)\in\cF_{+1}\setminus\{0\}}
	   \frac{\abs{\dual{\psi}{g}_\cF}}
	        {\norm{U(\cc{w_0})g}_{+\cF}},
\end{align*}
where the last equation arises from $U(w):\cF_{+1}\to\cF_{+1}$ being
onto (see \eqref{eq:cayley-transform}). Since $U(w)$ is also an
isometry in there,
\begin{equation*}
\norm{\frac{(\cdot)-\cc{w_0}}{(\cdot)-w_0}f(\cdot)}_{-\cF}
	= \norm{\hat{U}(\cc{w_0})\psi}_{-\cF}
	= \norm{\psi}_{-\cF}
	= \norm{f}_{-\cF},
\end{equation*}
thereby completing the proof.
\end{proof}

In view of Proposition~\ref{prop:map-f-assoc} and
Theorem~\ref{thm:assoc-with-norm-is-dB}, we henceforth identify
$\cF_{-1}$ with its realization as the Hilbert space $\assoc\cB$
which is equipped with the norm \eqref{eq:norm-for-assoc}. Following up this
identification, we define
\begin{equation*}
\dual{f}{g}_\cF := \dual{\psi}{g}_\cF,
\quad f(z)\in\assoc\cB,
\quad g(z)\in\dom(S^*),
\end{equation*}
where $\psi\in\cF_{-1}$ is related to $f(z)$ by
\eqref{eq:F-minus-one-as-assoc}. In particular,
\begin{equation*}
\dual{k(\cdot,z)}{f}_\cF = f(z),\quad f(z)\in\cF_{-1}.
\end{equation*}
The next result provides an explicit formula for the duality pairing
between $\cF_{+1}$ and $\cF_{-1}$.

\begin{proposition}
  Suppose $\psi\in\cF_{-1}$ and $f(z)\in\assoc\cB$ be such that
  $f(z)=\dual{k(\cdot,z)}{\psi}_\cF$. Then, for every
  $g(z)\in\cF_{+1}$,
\begin{equation}
\label{eq:duality-for-K}
\dual{\psi}{g}_{\cF}
	= \int_\R \frac{\cc{f(x)}h(x)}{\abs{e(x)}^2}dx +
		a^{(+)}f^{\#}(i) + a^{(-)}f^{\#}(-i),
\end{equation}
where $g(z)=h(z) + a^{(+)}k(z,-i) + a^{(-)}k(z,i)$ is decomposed in
accordance with \eqref{eq:adjoint-S-dom}.
\end{proposition}
\begin{proof}
  We have $\cF_{-1}=\cB(e_\cF)$ for some $e_\cF(z)\in\cHB$. Also,
  since $\cB=\cB(e)$ for some $e(z)\in\cHB$ free of real zeros (see
  Remark~\ref*{rem:HB-no-real-zeros}), $\cF_{-1} = \cB(e_1)$ setwise
  with $e_1(z):=(z+i)e(z)$. Thus, if $\fI(g)$ denotes the right hand side of
\eqref{eq:duality-for-K}, then
\begin{align*}
\abs{\fI(g)}
	&\le\norm{f}_{\cB(e_1)}\norm{h}_{+2}
	   +\abs{a^{(+)}}\norm{f}_{-\cF}\norm{k(\cdot,-i)}_{+\cF}
	   +\abs{a^{(-)}}\norm{f}_{-\cF}\norm{k(\cdot,i)}_{+\cF}
	\\[1mm]
	&\le C\norm{f}_{-\cF}\left(\norm{h}_{+2}
		   +\left(\abs{a^{(+)}}+\abs{a^{(-)}}\right)
		    \norm{k(\cdot,i)}_{+\cF}\right),
\end{align*}
where the second inequality is implied by the fact that the norms
induced by $e_\cF$ and $e_1$ are equivalent. Now, since
\begin{equation*}
\norm{g}_{+\cF}^2
	= \norm{h}_{+2}^2 + \left(\abs{a^{(+)}}^2+\abs{a^{(-)}}^2\right)
		\norm{k(\cdot,i)}_{+\cF}^2,
\end{equation*}
it follows that, for some possibly different $C>0$,
\begin{equation*}
\abs{\inner{f}{g}_{\cF}}
	\le C\norm{f}_{-\cF}\norm{g}_{+\cF}
\end{equation*}
so $\fI$ defines a continuous functional on $\cF_{+1}$. Finally,
it is easy to verify that $\fI$ coincides with $\inner{f}{\cdot}$
whenever $f(z)\in\cB$ so it is equal to $\dual{\psi}{\cdot}_\cF$.
\end{proof}

Let $\cM_0^{(\gamma)}\subset\cF_{-1}$ be the annihilator of
$\dom(S_\gamma)$, that is,
\begin{equation*}
\cM_0^{(\gamma)}
	:=\left\{g(z)\in\assoc\cB:\dual{g}{h}_\cF=0
	         \text{ for all }h(z)\in\dom(S_\gamma)\right\}.
\end{equation*}

\begin{theorem}
\label{thm:b-minus-2-is-not-dB}
Every element $\fF\in\assoc\cB/\cM_0^{(\gamma)}$ defines a unique element in
$\cB_{-2}^{(\gamma)}$ by the rule
\begin{equation}
\label{eq:description-of-B-2}
\dual{\fF}{h}_2
	:= \dual{f}{h}_\cF,\quad h(z)\in\cF_{+1},
\end{equation}
where $f(z)\in\assoc\cB$ is any representative of $\fF$. Conversely,
to every $\varphi\in\cB_{-2}^{(\gamma)}$ there corresponds a unique
$\fF\in\assoc\cB/\cM_0^{(\gamma)}$ such that
\begin{equation*}
\dual{\varphi}{h}_{2} = \dual{\fF}{h}_{2}
\end{equation*}
in the sense of \eqref{eq:description-of-B-2}.
\end{theorem}
\begin{proof}
  The trueness of the first part of the statement is rather obvious,
  since for every element $g(z)$ of $\dom(S_\gamma)$ one has
  $\norm{g}_{+2}=\norm{g}_{+\cF}$. Indeed,
\begin{equation*}
  \norm{\fF}_{-2}
  = \sup_{g(z)\in\dom(S_\gamma)\setminus\{0\}}
  \frac{\abs{\dual{f}{g}_\cF}}{\norm{g}_{+\cF}}
  \le \norm{f}_{-\cF},
\end{equation*}
for all $f(z)$ within the equivalence class $\fF$.

As for the second part of the statement, consider
$\varphi\in\cB_{-2}^{(\gamma)}$. Given $c\in\C$, define
$\varphi_c\in\cF_{-1}$ by the rule
\begin{equation*}
\dual{\varphi_c}{g}_\cF
	:= \dual{\varphi}{h}_2 + \cc{c}b,
\end{equation*}
where $h(z)\in\dom(S_\gamma)$ and $b\in\C$ are related to $g(z)\in\dom(S^*)$
by the decomposition defining \eqref{eq:adjoint-S-usual-dom}
\cite[Lemma~1.3.1]{albeverio-kurasov}. By Proposition~\ref{prop:map-f-assoc},
$\varphi_c$ can be identified with some $f_c(z)\in\assoc\cB$. Clearly,
\begin{equation*}
\dual{f_c}{\cdot}_\cF\restriction_{\dom(S_\gamma)}
	= \dual{\varphi}{\cdot}_2.
\end{equation*}
Thus, the associated element in $\assoc\cB/\cM_0^{(\gamma)}$ is
$\fF=\left\{f_c(z):c\in\C\right\}$.
\end{proof}


\section{Singular rank-one perturbations}
\label{sec:singular-rank-one-perturbations}

Let us turn to singular rank-one perturbations of selfadjoint
extensions of the operator $S$. To keep the notation simple, let us
fix $\gamma=\pi/2$. We show below that the ``correct'' perturbation of
the operator $S_{\pi/2}$ is performed along the function $s_0(z)$ in
the sense that $S$ is precisely $S_{\pi/2}$ restricted to those
$f(z)\in\dom(S_{\pi/2})$ that obey $\dual{s_0}{f}_2=0$.

\begin{lemma}
Assume $\mu\in\spec(S_{\pi/2})$. Then,
\begin{equation*}
\dual{k(\cdot,\mu)}{g}_{2} = g(\mu)
\end{equation*}
for every $g(z)\in\assoc\cB$.
\end{lemma}
\begin{proof}
Given $g(z)\in\assoc\cB$, let $\{h_l(z)\}\subset\cB$ be any sequence converging
to it in $\cB_{-2}^{(\pi/2)}$.
Then, recalling that $k(z,\mu)\in\dom(S_{\pi/2})$, one has
\begin{equation*}
\abs{h_l(\mu)-h_m(\mu)}
	=   \abs{\dual{k(\cdot,\mu)}{(h_l-h_m)}_2}
	\le \norm{k(\cdot,\mu)}_{+2}\norm{h_l-h_m}_{-2}
\end{equation*}
so $\{h_l(\mu)\}$ is convergent. By standard arguments the limit equals
$g(\mu)$.
\end{proof}

\begin{lemma}
\label{lem:resolvent-on-s}
For every $w\not\in\spec(S_{\pi/2})$,
\begin{equation*}
\hat{R}_{\pi/2}(w)s_0(z) = -\frac{\pi}{s_{\pi/2}(w)}k(z,\cc{w}).
\end{equation*}
\end{lemma}
\begin{proof}
We need to show that
\begin{equation*}
\inner{\hat{R}_{\pi/2}(w)s_0(\cdot)}{f(\cdot)}
	= -\frac{\pi}{s_{\pi/2}(\cc{w})}f(\cc{w})
\end{equation*}
for every $f(z)\in\cB$. Due to the continuity of the inner product, it will
suffice to show the assertion on elements of a basis of $\cB$. Thus, let
us consider $k(z,\mu_n)$ with $\mu_n\in\spec(S_{\pi/2})$. We have
\begin{align*}
\inner{\hat{R}_{\pi/2}(w)s_0(\cdot)}{k(\cdot,\mu_n)}
	&= \inner{s_0(\cdot)}{(S_{\pi/2} - \cc{w} I)^{-1}k(\cdot,\mu_n)}_{2}
\\[1mm]
	&= \inner{s_0(\cdot)}{(\mu_n - \cc{w})^{-1}k(\cdot,\mu_n)}_{2}
\\[1mm]
	&= (\mu_n - \cc{w})^{-1}\inner{s_0(\cdot)}{k(\cdot,\mu_n)}_{2}
\\[1mm]
	&= (\mu_n - \cc{w})^{-1}s_0(\mu_n),
\end{align*}
where we have used the previous lemma and the fact that $s_0(z)$ is real entire.

Recalling \eqref{eq:k-given-by-s} we have
\begin{equation*}
k(\cc{w},\mu_n)= \frac{s_{\pi/2}(\cc{w})s_0(\mu_n)}{\pi(\cc{w}-\mu_n)},
\end{equation*}
whence the assertion follows.
\end{proof}

\begin{lemma}
The function $s_{0}(z)$ is a generating element (cyclic in the
terminology of \cite{gesztesy-simon}) of $S_{\pi/2}$.
\end{lemma}
\begin{proof}
Since $S$ is simple and $k(z,\cc{w})\in\ker(S^*- w I)$, $\cB$ is the closure of
\begin{equation*}
\lspan_{w\in\C\setminus\R}\left\{\ker(S^*-wI)\right\}
	= \lspan_{w\in\C\setminus\R}\left\{k(z,\cc{w})\right\}.
\end{equation*}
This identity along with Lemma~\ref{lem:resolvent-on-s} implies that
\begin{equation*}
\lspan_{w\in\C\setminus\R}\left\{\hat{R}_{\pi/2}(w)s_0(z)\right\}
\end{equation*}
is a total set in $\cB$.
\end{proof}

\begin{proposition}
\label{prop:S-as-restriction}
Define $\cD_0:=\{f(z)\in\dom(S_{\pi/2}):\dual{s_0}{f}_2=0\}$.
Then,
\begin{equation*}
S = S_{\pi/2}\restriction_{\cD_0}.
\end{equation*}
\end{proposition}
\begin{proof}
First, choose any $f(z)\in\dom(S)$ and set $g(z):=(z+i)f(z)$; 
clearly $g(z)\in\cB$. Then, according to \eqref{eq-2pairing-expression},
\begin{align*}
\dual{s_0}{f}_{2}
	&= \inner{\hat{R}_{\pi/2}(i)s_0(\cdot)}{(S_{\pi/2} + iI)f(\cdot)}
	\\[1mm]
	&= -\frac{\pi}{s_{\pi/2}(-i)}\inner{k(\cdot,-i)}{g(\cdot)}
	\\[1mm]
	&= g(-i) = 0.
\end{align*}
Therefore, $\dom(S)\subset\cD_0$.

Now suppose $f(z)\in\cD_0$. Since $f(z)\in\dom(S_{\pi/2})$, it follows
from \eqref{eq:definition-selfadj} that
\begin{equation*}
(S_{\pi/2} + iI)f(z)
	= (z+i)f(z) + \frac{s_{\pi/2}(z)}{s_{\pi/2}(-i)}g(-i)\,,
\end{equation*}
where $g(z)\in\cB$ satisfies
\begin{equation*}
f(z) = \frac{g(z)-\frac{s_{\pi/2}(z)}{s_{\pi/2}(-i)}g(-i)}{z+i}.
\end{equation*}
A computation like the one above yields
\begin{align*}
\dual{s_0}{f}_{2}
	 = \inner{\hat{R}_{\pi/2}(i)s_0(\cdot)}{(S_{\pi/2} + iI)f(\cdot)}
	 = g(-i)
\end{align*}
so the assumption $\dual{s_0}{f}_{2}=0$ implies $g(-i)=0$
in turn implying $f(z)\in\dom(S)$. This completes the proof.
\end{proof}

Proposition \ref{prop:S-as-restriction} implies that any other selfadjoint
$S_\gamma$ of $S$ is related to $S_{\pi/2}$ by  Krein's formula
\cite[Theorem~1.2.1]{albeverio-kurasov},
\begin{equation}
\label{eq:krein-formula}
R_{\gamma}(w) - R_{\pi/2}(w)
	= \frac{1}{\lambda-q(w)}\inner{\hat{R}_{\pi/2}(\cc{w})s_0(\cdot)}{\cdot}
	  \hat{R}_{\pi/2}(w)s_0(z)
\end{equation}
with $w:\im(w)\ne 0$, where
\begin{equation}
\label{eq:function-q}
q(w) := \inner{\hat{R}_{\pi/2}(i)s_0(\cdot)}
		{(I+wS_{\pi/2})R_{\pi/2}(w)\hat{R}_{\pi/2}(i)s_0(\cdot)}
\end{equation}
is the Krein's $Q$-function. Below we find the relation between
$\lambda$ and $\gamma$.

\begin{lemma}
\label{lem:on-function-q}
Let $q(w)$ be the function defined by \eqref{eq:function-q}. Then,
\begin{equation}
\label{eq:identity-for-q}
	q(w) = \pi\re\left(\frac{s_0(i)}{s_{\pi/2}(i)}\right) -
		   \pi\frac{s_0(w)}{s_{\pi/2}(w)}.
\end{equation}
\end{lemma}
\begin{proof}
Since $(I+wS_{\pi/2})(S_{\pi/2}-wI)^{-1}=wI +
(w^2+1)(S_{\pi/2}-wI)^{-1}$, it follows from
Lemma~\ref{lem:resolvent-on-s} that
\begin{equation*}
q(w) = \frac{\pi^2}{\abs{s_{\pi/2}(i)}^2}wk(i,i) +
       \frac{\pi^2}{\abs{s_{\pi/2}(i)}^2}(w^2+1)
       \inner{k(\cdot,-i)}{(S_{\pi/2}-wI)^{-1}k(\cdot,-i)}.
\end{equation*}
Resorting to Lemma~\ref{lem:quotient-diff-reprod-kernel}, one rewrites
this as follows
\begin{equation*}
q(w) = \frac{\pi^2}{\abs{s_{\pi/2}(i)}^2}(w+i)
	   \frac{s_{\pi/2}(i)}{s_{\pi/2}(w)}k(w,i) -
	   \frac{\pi^2}{\abs{s_{\pi/2}(i)}^2}ik(i,i).
\end{equation*}
Finally, a computation involving \eqref{eq:k-given-by-s} yields
\eqref{eq:identity-for-q}.
\end{proof}

\begin{proposition}
\label{lem:quasi-krein}
For $w\in\C\setminus\R$ one has
\begin{equation*}
R_{\gamma}(w) - R_{\pi/2}(w)
	= \frac{1}{\pi\tan\gamma+\pi\frac{s_{0}(w)}{s_{\pi/2}(w)}}
	  \inner{\hat{R}_{\pi/2}(\cc{w})s_0(\cdot)}{\cdot}
	  	  \hat{R}_{\pi/2}(w)s_0(z).
\end{equation*}
Consequently, Krein's formula \eqref{eq:krein-formula} holds true for
$\lambda = \pi\tan\gamma + \pi\re\left(\frac{s_0(i)}{s_{\pi/2}(i)}\right)$.
\end{proposition}
\begin{proof}
  Let $f(z)\in\cB$. By using \eqref{eq:s-beta},
  \eqref{eq:dom-selfadj}, and Lemma~\ref{lem:resolvent-on-s} one
  obtains
\begin{align*}
(R_{\gamma}(w) - R_{\pi/2}(w))f(z)
&=   \frac{s_{\pi/2}(z)s_{\gamma}(w)-s_{\pi/2}(w)s_{\gamma}(z)}{z-w}
     \frac{f(w)}{s_{\pi/2}(w)s_{\gamma}(w)}
     \\
&=   \pi (\cos\gamma) k(z,\cc{w}) \frac{f(w)}{s_{\pi/2}(w)s_{\gamma}(w)}
	 \\[1mm]
&= - \frac{\cos\gamma}{s_\gamma(w)}f(w)\hat{R}_{\pi/2}(w) s_0(z)
	 \\[1mm]
&= 	 \frac{\cos\gamma}{\pi}\frac{s_{\pi/2}(w)}{s_\gamma(w)}
     \inner{\hat{R}_{\pi/2}(\cc{w})s_0(\cdot)}{f(\cdot)}
     		\hat{R}_{\pi/2}(w) s_0(z),
\end{align*}
whence the first statement follows. The second assertion is a consequence of
Lemma~\ref{lem:on-function-q}.
\end{proof}

\begin{theorem}
  \label{teo:dom-selfadj-alternative}
For $\beta\in[0,\pi)\setminus\{\pi/2\}$, the following
characterization holds
\begin{equation*}
\dom(S_\gamma) =
\left\{\begin{gathered}
		g(z) = h(z) + b S_{\pi/2}R_{\pi/2}(-i)\hat{R}_{\pi/2}(i)s_0(z),
		\\
		h(z)\in\dom(S_{\pi/2}), b\in\C:\dual{s_0}{h}_2 =
		\pi b\left(\tan\gamma+\re\frac{s_{0}(i)}{s_{\pi/2}(i)}\right)
       \end{gathered}\right\}.
\end{equation*}
\end{theorem}
\begin{proof}
Let $\cD$ denote the right hand side of the claimed identity.
First, suppose $g(z)\in\dom(S_\gamma)$. According to
\eqref{eq:dom-selfadj}, for some unique $f(z)\in\cB$, one has
\begin{align*}
g(z)  = \frac{f(z)-\frac{s_{\gamma}(z)}{s_{\gamma}(i)}f(i)}{z-i}
	  = \frac{f(z)-\frac{s_{\pi/2}(z)}{s_{\pi/2}(i)}f(i)}{z-i}
		+
		\frac{\frac{s_{\pi/2}(z)}{s_{\pi/2}(i)} -
		\frac{s_{\gamma}(z)}{s_{\gamma}(i)}}{z-i}f(i).
\end{align*}
The first term above belongs to $\dom(S_{\pi/2})$. As for the second
term, we have
\begin{align*}
\text{2nd term}
	&= \frac{\cos\gamma}{s_{\pi/2}(i)s_\gamma(i)}
	   \frac{s_{\pi/2}(z)s_0(i)-s_0(z)s_{\pi/2}(i)}{z-i}f(i)
	\\[1mm]
	&= \frac{\cos\gamma}{s_{\pi/2}(i)s_\gamma(i)}\pi k(z,-i)f(i)
	\\[1mm]
	&= - \frac{\cos\gamma}{s_\gamma(i)}f(i)\hat{R}_{\pi/2}(i)s_0(z);
\end{align*}
in this derivation we use \eqref{eq:k-given-by-s} and
Lemma~\ref{lem:resolvent-on-s}. Now define
\begin{equation*}
h(z) := \frac{f(z)-\frac{s_{\pi/2}(z)}{s_{\pi/2}(i)}f(i)}{z-i}
		- i \frac{\cos\gamma}{s_\gamma(i)}f(i)R_{\pi/2}(-i)
			\hat{R}_{\pi/2}(i)s_0(z);
\end{equation*}
clearly $h(z)\in\dom(S_{\pi/2})$. Then,
\begin{align*}
g(z) &= h(z) - \frac{\cos\gamma}{s_\gamma(i)}f(i)
	    \left[\hat{R}_{\pi/2}(i)s_0(z)
	    	- iR_{\pi/2}(-i)\hat{R}_{\pi/2}(i)s_0(z)\right]
	    \\[1mm]
	 &= h(z) + b S_{\pi/2}R_{\pi/2}(-i)\hat{R}_{\pi/2}(i)s_0(z),\quad
        b := - \frac{\cos\gamma}{s_\gamma(i)}f(i).
\end{align*}
Moreover,
\begin{align*}
\dual{s_0}{h}_2
	&= \inner{\hat{R}_{\pi/2}(-i)s_0(\cdot)}{f(\cdot)}
	   + i b \inner{\hat{R}_{\pi/2}(i)s_0(\cdot)}{\hat{R}_{\pi/2}(i)s_0(\cdot)}
	   \\[1mm]
	&= - \frac{\pi}{s_{\pi/2}(i)}f(i)
	   + i b \frac{\pi}{\abs{s_{\pi/2}(i)}^2}k(i,i)
	   \\[1mm]
	&= \pi b \left(\frac{1}{\cos\gamma}\frac{s_\gamma(i)}{s_{\pi/2}(i)}
		+i\frac{\frac{s_0(-i)}{s_{\pi/2}(-i)}-\frac{s_0(i)}{s_{\pi/2}(i)}}
		{2i}\right)
	   \\[1mm]
	&= \pi b\left(\tan\gamma+\re\frac{s_{0}(i)}{s_{\pi/2}(i)}\right).
\end{align*}

So far we have shown that $\dom(S_\gamma)\subset\cD$.
To prove that $\cD$ is contained in $\dom(S_\gamma)$, it suffices to show that, for
every $g(z)\in\cD$, one has $g(z)=R_\gamma(i)l(z)$ for some $l(z)\in\cB$. So
let $g(z)\in\cD$, i.\,e.,
\begin{equation*}
g(z) = h(z) + b S_{\pi/2}R_{\pi/2}(-i)\hat{R}_{\pi/2}(i)s_0(z)
\end{equation*}
with $h(z)=R_{\pi/2}(i)f(z)$ and $b\in\C$ such that
\begin{equation}
\label{eq:boring}
\dual{s_0}{h}_2
	=  \dual{\hat{R}_{\pi/2}(-i)s_0}{f}
	= -\frac{\pi}{s_{\pi/2}(i)}f(i)
	=  \pi b\left(\tan\gamma+\re\frac{s_{0}(i)}{s_{\pi/2}(i)}\right).
\end{equation}
Since $S_{\pi/2}R_{\pi/2}(-i) = I - iR_{\pi/2}(-i)$ and
$R_{\pi/2}(-i)\hat{R}_{\pi/2}(i)=R_{\pi/2}(i)\hat{R}_{\pi/2}(-i)$,
one has
\begin{equation}
\label{eq:from-dom-to-dom-1}
g(z) = R_{\pi/2}(i) \left(f(z) - i b\hat{R}_{\pi/2}(-i)s_0(z)\right)
	   +  b \hat{R}_{\pi/2}(i)s_0(z).
\end{equation}
Due to Lemma~\ref{lem:resolvent-on-s} and Proposition~\ref{lem:quasi-krein},
the first term above becomes
\begin{multline*}
\text{1st term}
	= R_{\gamma}(i) \left(f(z) + i\pi b\frac{k(z,i)}{s_{\pi/2}(-i)}\right)
	\\
	- \frac{1}{\tan\gamma+\frac{s_{0}(i)}{s_{\pi/2}(i)}}
	  \inner{\hat{R}_{\pi/2}(-i)s_0(\cdot)}{f(\cdot)
	- i b\hat{R}_{\pi/2}(-i)s_0(\cdot)}\hat{R}_{\pi/2}(i)s_0(z).
\end{multline*}
However,
\begin{align*}
\inner{\hat{R}_{\pi/2}(-i)s_0(\cdot)}{f(\cdot)
	- i b\hat{R}_{\pi/2}(-i)s_0(\cdot)}
&= -\frac{\pi}{s_{\pi/2}}f(i) - ib\frac{\pi^2}{s_{\pi/2}(i)s_{\pi/2}(-i)}k(i,i)
	\\[1mm]
&= \dual{s_0}{h}_2 + i\pi b \im\frac{s_0(i)}{s_{\pi/2}(i)}
	\\[1mm]
&= \pi b \tan\gamma + \pi b \frac{s_0(i)}{s_{\pi/2}(i)}.
\end{align*}
Thus,
\begin{equation}
\label{eq:from-dom-to-dom-2}
\text{1st term}
= R_{\gamma}(i) \left(f(z) + i\pi b\frac{k(z,i)}{s_{\pi/2}(-i)}\right)
	- b \hat{R}_{\pi/2}(i)s_0(z),
\end{equation}
which shows that $g(z)\in\dom(S_{\gamma})$.
\end{proof}

\begin{lemma}
For every $g(z)\in\dom(S^*)$,
\begin{equation}
\label{eq:sigma_0}
\dual{s_0}{g}_\cF = \dual{s_0}{h}_2
	- \pi b \re\frac{s_{0}(i)}{s_{\pi/2}(i)},
\end{equation}
with $h(z)\in\dom(S_{\pi/2})$ and $b\in\C$ related to $g(z)$ by the
(unique) decomposition
$g(z) = h(z) + b S_{\pi/2}R_{\pi/2}(-i)\hat{R}_{\pi/2}(i)s_0(z)$.
\end{lemma}
\begin{proof}
An obvious computation shows that
\begin{equation*}
S_{\pi/2}R_{\pi/2}(-i)\hat{R}_{\pi/2}(i)
	= \frac12\hat{R}_{\pi/2}(-i) + \frac12\hat{R}_{\pi/2}(i)
\end{equation*}
Then, recalling Lemma~\ref{lem:resolvent-on-s},
\begin{multline*}
\dual{s_0}{S_{\pi/2}R_{\pi/2}(-i)\hat{R}_{\pi/2}(i)s_0}_\cF
\\[1mm]
\begin{aligned}
	&= -\frac{\pi}{2s_{\pi/2}(-i)}\dual{s_0}{k(\cdot,i)}_\cF
	  -\frac{\pi}{2s_{\pi/2}(i)}\dual{s_0}{k(\cdot,-i)}_\cF
	\\[1mm]
	&= -\pi\re\frac{s_{0}(i)}{s_{\pi/2}(i)}.
\end{aligned}
\end{multline*}
The assertion
now follows straightforwardly.
\end{proof}

For the statement of our last result below, we recall that
$S_{\pi/2}:\cB_{+2}^{(\pi/2)}\to\cB$ has a unique contractive continuation
$\hat{S}_{\pi/2}:\cB\to\cB_{-2}^{(\pi/2)}$.

\begin{theorem}
\label{thm:action-s-gamma}
Let $\tilde{S}_{\pi/2}$ be the restriction of $\hat{S}_{\pi/2}$ to
$\dom(S^*)$. For $\gamma\in(0,\pi)$, define 
$\tilde{S}_\gamma:\dom(S^*)\to\cB_{-2}^{(\pi/2)}$ by
\begin{equation*}
\tilde{S}_\gamma := \tilde{S}_{\pi/2}
				- \frac{\cot\gamma}{\pi}\dual{s_0}{\cdot}_{\cF}s_0(z).
\end{equation*}
Then, $S_\gamma$ as described by \eqref{eq:definition-selfadj} is the
restriction of $\tilde{S}_\gamma$ to $\dom(S_\gamma)$.
\end{theorem}
\begin{proof}
If $g(z)\in\dom(S_\gamma)$, then
\begin{equation*}
g(z) = h(z) + b S_{\pi/2}R_{\pi/2}(-i)\hat{R}_{\pi/2}(i)s_0(z)
\end{equation*}
with $h(z)\in\dom(S_{\pi/2})$ and $b\in\C$, both related to each other
in accordance with
Theorem~\ref{teo:dom-selfadj-alternative}. Taking into account
\eqref{eq:sigma_0}, one obtains
\begin{equation*}
\dual{s_0}{g}_\cF
	= \dual{s_0}{h}_{2} - \pi b \re\frac{s_{0}(i)}{s_{\pi/2}(i)}
	= \pi b \tan\gamma.
\end{equation*}
Thus, so far,
\begin{equation}
\label{eq:first-reduction}
\tilde{S}_\gamma g(z)
	= S_{\pi/2}h(z)
	  + b\hat{S}_{\pi/2}S_{\pi/2}R_{\pi/2}(-i)\hat{R}_{\pi/2}(i)s_0(z)
	  - b s_0(z).
\end{equation}
Since
\begin{equation*}
\hat{S}_{\pi/2}S_{\pi/2}R_{\pi/2}(-i)\hat{R}_{\pi/2}(i)
	= \hat{I} - R_{\pi/2}(-i)\hat{R}_{\pi/2}(i)
\end{equation*}
(here we use that $(\hat{S}_{\pi/2} -iI)\hat{R}_{\pi/2}(i)=\hat{I}$,
the latter being the identity operator in $\cB_{-2}^{(\pi/2)}$), equation
\eqref{eq:first-reduction}
becomes
\begin{align*}
\tilde{S}_\gamma g(z)
	&= S_{\pi/2}h(z) - b R_{\pi/2}(-i)\hat{R}_{\pi/2}(i) s_0(z)
	\\[1mm]
	&= S_{\pi/2}h(z) + i\frac{\pi b}{2}
	   \left(\frac{k(z,i)}{s_{\pi/2}(-i)}-\frac{k(z,-i)}{s_{\pi/2}(i)}\right);
\end{align*}
the last equality follows from the resolvent identity
\eqref{eq:ext-res-identity} and Lemma~\ref{lem:resolvent-on-s}.
On the other hand, since $h(z)=R_{\pi/2}(i)f(z)$ for some $f(z)\in\cB$, one
has (see \eqref{eq:from-dom-to-dom-1} and \eqref{eq:from-dom-to-dom-2})
\begin{equation*}
g(z) = R_{\gamma}(i)\left(f(z) + i\pi b\frac{k(z,i)}{s_{\pi/2}(-i)}\right)\,.
\end{equation*}
Therefore,
\begin{align*}
S_\gamma g(z)
	&=  z h(z) + zb S_{\pi/2}R_{\pi/2}(-i)\hat{R}_{\pi/2}(i)s_0(z)
	\\[1mm]
	&\phantom{= zh(z)\ }
	    + \frac{s_{\gamma}(z)}{s_{\gamma}(i)}
		    \left(f(i) + i\pi b\frac{k(i,i)}{s_{\pi/2}(-i)}\right)
	\\[1mm]
	&= S_{\pi/2} h(z)
		- z \frac{\pi b}{2}\left(\frac{k(z,i)}{s_{\pi/2}(-i)}
		+   \frac{k(z,-i)}{s_{\pi/2}(i)}\right)
	\\[1mm]
	 &\phantom{= S_{\pi/2} h(z)\ } + \left(\frac{s_{\gamma}(z)}{s_{\gamma}(i)}
	   			- \frac{s_{\pi/2}(z)}{s_{\pi/2}(i)}\right)f(i)
	   			+ i\pi b \frac{k(i,i)}{s_{\pi/2}(-i)}
	   			  \frac{s_{\gamma}(z)}{s_{\gamma}(i)}.
\end{align*}
Thus, it suffices to show that
\begin{multline}
\label{eq:horror}
0 =  - i\frac{\pi b}{2}
	   \left(\frac{k(z,i)}{s_{\pi/2}(-i)}-\frac{k(z,-i)}{s_{\pi/2}(i)}\right)
	 - z \frac{\pi b}{2}\left(\frac{k(z,i)}{s_{\pi/2}(-i)}
	     		+   \frac{k(z,-i)}{s_{\pi/2}(i)}\right)
	\\[1mm]
     + \left(\frac{s_{\gamma}(z)}{s_{\gamma}(i)}
     	   			- \frac{s_{\pi/2}(z)}{s_{\pi/2}(i)}\right)f(i)
     	   			+ i\pi b \frac{k(i,i)}{s_{\pi/2}(-i)}
     	   			  \frac{s_{\gamma}(z)}{s_{\gamma}(i)}.
\end{multline}
Note that
\begin{equation*}
\frac{s_{\gamma}(z)}{s_{\gamma}(i)}
	   			- \frac{s_{\pi/2}(z)}{s_{\pi/2}(i)}
= - \frac{\pi\cos\gamma}{s_\gamma(i)}(z-i)\frac{k(z,-i)}{s_{\pi/2}(i)}
\end{equation*}
so, in view of \eqref{eq:boring},
\begin{align*}
\left(\frac{s_{\gamma}(z)}{s_{\gamma}(i)}
	   			- \frac{s_{\pi/2}(z)}{s_{\pi/2}(i)}\right)\!f(i)
&= \pi b \frac{\tan\gamma + \re\frac{s_0(i)}{s_{\pi/2}(i)}}
			 {\tan\gamma + \frac{s_0(i)}{s_{\pi/2}(i)}}
			 (z-i)\frac{k(z,-i)}{s_{\pi/2}(i)}
\\
&= \pi b(z-i)\frac{k(z,\!-i)}{s_{\pi/2}(i)}
	- i\pi b \frac{\im\frac{s_0(i)}{s_{\pi/2}(i)}}
			      {\tan\gamma + \frac{s_0(i)}{s_{\pi/2}(i)}}
				  (z-i)\frac{k(z,\!-i)}{s_{\pi/2}(i)}.
\end{align*}
Also,
\begin{equation*}
i\pi b \frac{k(i,i)}{s_{\pi/2}(-i)}\frac{s_{\gamma}(z)}{s_{\gamma}(i)}
	= -i b\frac{s_\gamma(z)}{\cos\gamma}
	    \frac{\im\frac{s_0(i)}{s_{\pi/2}(i)}}
	    {\tan\gamma+\frac{s_0(i)}{s_{\pi/2}(i)}}.
\end{equation*}
Thus, the sum of the last two terms in \eqref{eq:horror} amounts to
\begin{align*}
\text{3rd term} +\text{4th term}
	&= \pi b(z-i)\frac{k(z,\!-i)}{s_{\pi/2}(i)}
		- ib\im\frac{s_0(i)}{s_{\pi/2}(i)}s_{\pi/2}(z).
	\\[1mm]
	&= b\left(s_{\pi/2}(z)\re\frac{s_0(i)}{s_{\pi/2}(i)}-s_0(z)\right).
\end{align*}
Additionally, a rather tedious computation shows that the first two terms yield
\begin{equation*}
\text{1st term}+\text{2nd term}
	= - b\left(s_{\pi/2}(z)\re\frac{s_0(i)}{s_{\pi/2}(i)}-s_0(z)\right),
\end{equation*}
thus completing the proof.
\end{proof}

\paragraph*{Acknowledgments.}

Part of this work was done while J. H. T. visited IIMAS--UNAM 
(Mexico) in the winter of 2016. He deeply thanks
them for their kind hospitality.

\end{document}